\documentclass[12pt,reqno]{amsart}
\usepackage{amssymb}

\usepackage[curve,matrix,arrow]{xy}

\baselineskip

\theoremstyle{plain} 
\numberwithin{equation}{section}

\theoremstyle{definition}

\newtheorem*{Th}{Theorem}
\newtheorem*{Prop}{Proposition}

\theoremstyle{remark}

\def\CC{{\mathcal{C}}}
\def\CD{{\mathcal{D}}}

\def\CK{{\mathcal{K}}}

\def\CU{{\mathcal{U}}}

\def\Dim{\operatorname{Dim}\nolimits}

\def\HHH{\operatorname{H}\nolimits}

\def\HH#1#2#3{\HHH^{#1}(#2,#3)}

\def\Homul{\operatorname{\underline{Hom}}\nolimits}

\def\Ext{\operatorname{Ext}\nolimits}

\def\hgs{\HH{*}{G}{k}}

\title[Complexes over finite groups]
{The ranks of homology of complexes of projective modules over finite groups}
\author[Jon F. Carlson]{Jon F. Carlson}
\address{Department of Mathematics, University of Georgia, 
Athens, GA 30602, USA}
\email{jfc@math.uga.edu}
\thanks{Research partially supported by 
Simons Foundation grant 054813-01}

\date\today

\begin{document}

\begin{abstract} 
We show that counterexamples of Iyengar and Walker to the algebraic version
of Gunnar Carlsson's conjecture on the rank of the homology of a free complex
can be extended to examples over any finite group with many choices of 
the complex. 
\end{abstract}

\maketitle

Let $p$ be a prime.
A conjecture of Gunnar Carlsson \cite{GC} 
says that if an elementary abelian $p$-group
of rank $r$ acts freely on a CW-complex, then the sum of the dimensions of
the homology groups is at least $2^r$. The conjecture 
is still open. An algebraic version of the conjecture says that for 
such group $G$, and for any finite dimensional complex of free $kG$-modules
the sum of the dimensions of the homology groups of the complex is 
at least $2^r$. In a recent paper \cite{IW}, Iyengar and Walker gave a 
counterexample to this conjecture in the form of a cone over an 
endomorphism of the Koszul complex over $kG$. 

The purpose of this note is to point out that the Iyengar-Walker 
calculation combined with a construction of complexes of modules 
in a paper \cite{BC}  by Benson and this author gives a large number of 
counterexamples to the algebraic conjecture. Moreover, the counterexamples
are not cofined to the category of modules over elementary abelian groups
but can be constructed over any finite groups. Specifically, we 
prove the following. 

\begin{Th}
Suppose that $k$ is a field of odd characteristic $p \geq 3$ and that 
$G$ is a finite group having $p$-rank $r \geq 8$. Then there exists
an infinite number of mutually nonisomorphic 
finite dimensional complexes $\CD_*$ 
of projective $kG$-modules with the property that 
$\sum \Dim_k(\HHH_i(\CD_*)) < 2^r$. In addition, if $G$ is a $p$-group, then 
$\CD_*$ is a complex of free modules. 
\end{Th}

\begin{proof}
The proof consists of two parts. The first is the construction of the 
complexes following \cite{BC}. The second in the calculation of the dimension
of the homology, which is described in detail in \cite{IW}. Because each part
is well developed in the literature, we give only a sketch of the proof.

As for notation, we emphasize that $\otimes$ means $\otimes_k$ unless 
otherwise indicated. If $M$ and $N$ are $kG$-modules then the tensor product
$M \otimes N$ is a $kG$-module using the diagonal coproduct $g \mapsto 
g \otimes g$ for $g \in G$. 

Let $\zeta_1, \dots, \zeta_r$ be a homogeneous set of parameters
for the cohomology ring $\hgs$. That is, 
the $p$-rank $r$ of $G$ is also the Krull dimension of the 
cohomology ring, and $\hgs$ is finitely generated as a module over the 
subring generated by $\zeta_1, \dots, \zeta_r$. In addition, assume that
there is a positive integer $n$ such that 
the degree of $\zeta_i$ is $n$ for all $i$. 
Such a choice can be made as long as $n$ is a common multiple of the 
degrees of some collection of generators of $\hgs$. 
Note that since $p$ is odd, odd degree elements of $\hgs$ are 
nilpotent, and hence, it must be that $n$ is even. 

We construct a complex as follows. Assume 
that $(P_*, \varepsilon)$ is a minimal 
projective $kG$-resolution of the trivial module $k$. Each $\zeta_i$ is 
identified with a unique cocycle $\zeta_i: P_n \to k$. So we can
construct a diagram:
\[
\xymatrix{
\dots \ar[r] & P_{n+1} \ar[r] & P_n \ar[r] \ar[d]^{\zeta_i} & 
P_{n-1} \ar[r] \ar[d] & P_{n-2} \ar[r] \ar@{=}[d] &\dots \ar[r] 
& P_0 \ar[r]^{\varepsilon} \ar@{=}[d]& k \ar[r] \ar@{=}[d] & 0 \\
{} & 0 \ar[r] & k \ar[r]^{\iota_i} & L_{\zeta_i} \ar[r] & 
P_{n-2} \ar[r] &\dots \ar[r] & 
P_0 \ar[r]^{\varepsilon} & k \ar[r] & 0
}
\]
where the $L_{\zeta_i}$ is the pushout, so that the bottom row of
the diagram is an exact sequence. 
Let $\CC^i_*$ be the complex 
\[
\xymatrix{
0 \ar[r] & L_{\zeta_i} \ar[r] &
P_{n-2} \ar[r] &\dots \ar[r] &
P_0 \ar[r] & 0
}
\]
Note that the homology of $\CC^i_*$ consists of two copies of $k$ in 
degrees $0$ and $n-1$. 

Let $\CC_* = \CC^1_* \otimes \dots \otimes \CC^r_*$. 
Using support varieties or the method of \cite[Theorem 4.1]{BC}, we see that 
$L_{\zeta_1} \otimes \dots \otimes L_{\zeta_r}$
is a projective $kG$-module. If $G$ is a $p$-group then it is a 
free module, because all projective modules over the group algebra
of a $p$-group are free. Thus $\CC_*$ is a complex of projective modules. 
Its homology is that of a space that is homotopic to a product of 
$r$ spheres of dimension $n-1$ with trivial $G$-action on homology. 

Now $\CC^i_*$ has a degree $n-1$ chain map $\vartheta_i$ that takes 
$P_0$ to $L_{\zeta_i}$ by $\iota_i\varepsilon$. The action on homology 
takes the $\HHH_0(\CC^i)$ isomorphically to 
$\HHH_{n-1}(\CC^i)$. This induces a chain map 
\[
\xymatrix{
\theta_i= 1 \otimes \dots \otimes 1 \otimes \vartheta_i
\otimes 1 \otimes \dots \otimes 1: \CC_* \ar[r] &  \CC_*.
} 
\]
Now notice that because $n-1$ is 
odd, we have that $\theta_i\theta_j = -\theta_j\theta_i$
for $i \neq j$. That is, switching the order of the composition is 
tantamount to chainging the ordering of the factors in the tensor product
of the complexes.

So the algebra of chain maps $R \subseteq \Homul_{kG}(\CC_*,\CC_*)$, 
generated by the $\theta_i$'s, is seen to be 
an exterior algebra. Moreover, $\HHH_*(\CC_*)$ is (with some adjusting 
for degrees) the left regular representation of $R$. 

Now we follow Iyengar-Walker (particularly, \cite[Cor. 2.3]{IW}). 
Assume that $r \geq 8$, and let $\theta = 
\theta_1\theta_2 + \theta_3\theta_4 + \theta_5\theta_6
+ \theta_7\theta_8$. Let $\CD_*$ be the cone over $\theta$ or 
equivalently, the third object in the triangle of $\theta$
in the derived category. 
Iyengar and Walker prove that its homology has dimension $2^r - 2^{r-6}$,
thus giving our desired result.
\end{proof}

The ``infinite number" claim in the 
statement of the theorem is proved by the fact that there is an infinite
number of choices for the degree $n$ of the elements in the homogeneous 
set of parameters. It is not certain, at this point, what happens if 
we change the parameters keeping the same $n$. It seems likely that we 
get nonisomorphic complexes. In addition,  the thing that is really important
is that the products $\vartheta_{2n-1}\vartheta_{2n}$, for $n = 1, \dots, 4$, 
have the same degree so that $\theta$ is a chain map. 

We emphasize that the theorem is not proved to hold in the case that $p=2$.
As Iyengar and Walker point out, this stems from the fact that in an 
exterior algebra in characteristic $2$, the squares of all elements in the 
radical are zero. In addition, it seems to indicate that if one were to 
attempt to lift the system (systems of parameters and all) 
to the rational integers, then the homology would necessarily have $2$-torsion. 
It bodes ill for attempts to get a counterexample to the topological 
version of the conjecture, using this machinery. 
Indeed, R\"uping and Stephan \cite{RS}
have proved that the counterexample of Iyengar-Walker does not lift to the 
topological setting.

There is the question of whether the Iyengar-Walker counterexample 
fits into the scheme presented here. The answer is that it does up to
quasi-isomorphism. 

\begin{Prop} \label{prop} 
Suppose that $G = \langle g_1, \dots, g_r \rangle$ is an elementary abelian
$p$-group of order $p^r$. The Koszul complex used by Iyengar and Walker to 
construct their counterexample is quasi-isomorphic to a complex $\CC_*$ 
constructed as above from a homogeneous set of parameters for $\hgs$. 
\end{Prop}

\begin{proof}
Assume that $X_i = g_i-1$ so that $kG \cong A_1 \otimes \dots \otimes A_r$
where $A_i \cong  k[X_i]/(X_i^p)$ is the subalgebra generated by $X_i$. 
Let $\CU_*^i$ be the complex $A_i \to A_i$ where the map is multiplication
by $X_i$. We regard $\CU_i$ as a complex of $kG$-modules which is annihilated
by $X_j$ for all $j \neq i$. Now form the diagram
\[
\xymatrix{
\dots \ar[r] & P_2 \ar[r] \ar[d]^{\zeta_i} & P_1 \ar[r] \ar[d] &
P_0 \ar[r] \ar@{=}[d] & k \ar[r] \ar@{=}[d] & 0 \\
0 \ar[r] & k \ar[r] \ar@{=}[d] & L_{\zeta_i} \ar[r] \ar[d] & 
P_0 \ar[r] \ar@{=}[d] & k \ar[r] \ar@{=}[d] &0 \\
0 \ar[r] & k \ar[r] & U_1^i \ar[r]^{X_i}  & U_0^i \ar[r]  & k \ar[r] & 0.
}
\]
Here $(P_*, \varepsilon)$ is the minimal projective $kG$-resolution of $k$,
and $\zeta_i$ is the unique cocycle in the cohomology class in 
$\Ext_{kG}(k,k) \cong \HHH^2(G,k)$ that is represented by the sequence
on the bottom row. Thus the chain map from the top row to the bottom,
must factor throught the middle row. This shows that there is a 
quasi-isomorphism from the complex $\CC^i$ formed from $\zeta_i$, as in 
the proof of the theorem, to the complex $\CU^i_*$. 
Thus, by tensoring, we get a 
quasi-isomorphism from $\CC_*$ to $\CU_* = \CU^1_* \otimes \dots \otimes 
\CU^r_*$. We note that there is an action of the exterior algebra
on $r$ generators by chain 
maps on $\CU_*$ defined very similarly to the action on the complex 
$C_*$ and that the 
quasi-isomorphism commutes with the action of the exterior algebra. 

The only thing left to finish the proof, is to show that the complex $U_*$
is quasi-isomorphic to the Koszul complex $\CK_*$ in the Iyengar-Walker 
counterexample. In fact, the identifications that we make between the two
complexes are so natural that it could really be said that the two are 
identical. The Koszul complex is the product $\CK_* = \CK^1_* 
\otimes_{kG} \dots
\otimes_{kG} \CK^r_*$ where $\CK^i$ is the complex
\[
\xymatrix{
0 \ar[r] & kG \ar[r]^{X_i} &  \ar[r] & 0 
}
\]
im degrees 1 and 0. For any choice $i_1, \dots, i_r$ of elements of the 
set $\{0,1\}$, we have that 
\[
\CK^1_{i_1} \otimes_{kG} \dots \otimes_{kG} \CK^r_{i_r} \cong kG \cong 
\CU^1_{i_1} \otimes \dots \otimes \CU^r_{i_r}
\]
and the maps, which consist of multiplications by the $X_i$'s is the 
same on both. The total differential on both complexes is given by the 
same rule. The complexes are naturally identified to be the same.
We check that the chian maps defining the action of the exterior algebra
are the same on both complexes. 
\end{proof} 

{\bf Acknowledgment.} \ Part of this paper was written while I was a 
virtual participant at the conference on 
``Rank Conjectures in Algebraic Topology and Commutative Algebra" 
at the Banff International Reseearh 
Station. I would like to thank BIRS and the organizers of the 
stimulating meeting. I also thank Mark Walker for helping me sort through the 
proof of the proposition in the paper.


\end{document}